\documentclass[11pt,letterpaper,reqno]{amsart}
\usepackage{amssymb}
\usepackage{amsmath}
\usepackage{amsthm}
\usepackage{amsfonts}
\usepackage{bbm}
\usepackage{enumitem} 
\usepackage{booktabs}
\usepackage{graphicx}
\usepackage{xcolor}
\usepackage[T1]{fontenc}
\usepackage{doi}
\usepackage{comment} 
\usepackage{hyperref}
\hypersetup{pdfstartview={FitH}}
\usepackage{bookmark}
\usepackage[capitalize,noabbrev]{cleveref}

\usepackage{listings}
\usepackage[most]{tcolorbox}

\newtheorem{theorem}{Theorem}[section]
\newtheorem{lemma}[theorem]{Lemma}

\newtheorem{corollary}[theorem]{Corollary}

\newtheorem{conjecture}[theorem]{Conjecture}
\theoremstyle{definition}

\newtheorem{remark}[theorem]{Remark}
\newtheorem{definition}[theorem]{Definition}

\newtheorem*{proposition*}{Proposition}
\newtheorem*{theorem*}{Theorem}
\newtheorem*{corollary*}{Corollary}

\DeclareMathOperator{\tr}{tr}
\DeclareMathOperator{\spec}{spec}

\newcommand{\norm}[1]{\left\lVert #1\right\rVert}
\newcommand{\QQ}{\mathbb{Q}}

\begin{document}

\title[Positive $p$-Energies and Laplacian-Type Spectra]{Path-Minimality for Positive $p$-Energies, Laplacian-Type Spectra, and Line Graphs}
\author[Y.~Liu]{Yinchen Liu}
\author[Q.~Tang]{Quanyu Tang}

\address{Institute for Interdisciplinary Information Sciences, Tsinghua University, Beijing 100084, P. R. China}
\email{liuyinch23@mails.tsinghua.edu.cn}

\address{School of Mathematics and Statistics, Xi'an Jiaotong University, Xi'an 710049, P. R. China}
\email{tangquanyu827@gmail.com}



\begin{abstract}
We derive several applications of the path-minimality theorem for adjacency
$p$-energy proved in the companion paper. First, we prove the sharp inequality
$$
        \mathcal E_p^+(G)\ge \mathcal E_p^+(P_n),
$$
where $P_n$ is the path on $n$ vertices, in three settings: connected bipartite graphs for every real $p\ge2$, all connected graphs for every odd integer $p\ge3$, and all connected graphs for $p=4$. Second, using subdivision graphs, we prove path-minimality for Laplacian and signless Laplacian-type spectral sums, including power sums, Estrada-type quantities, resolvent energies, and thresholded tails. Third, we prove an edge-count second-order stop-loss comparison for the signless Laplacian above the threshold $2$. This yields the sharp line-graph inequality
$$
        \mathcal E_p^+(\mathcal L(G))\ge \mathcal E_p^+(P_m)
$$
for every connected graph $G$ with $m$ edges and every real $p\ge2$.
\end{abstract}

\maketitle

\section{Introduction}

Throughout this paper, all graphs are finite, simple, undirected, and unweighted.  Let \(G\) be a graph with vertex set \(V(G)=\{v_1,\ldots,v_n\}\), and let \(A(G)\) be its \(n\times n\) adjacency matrix, whose \((i,j)\)-entry is \(1\) if \(v_iv_j\in E(G)\) and is \(0\) otherwise.  We write
\[
        \lambda_1(G),\lambda_2(G),\ldots,\lambda_n(G)
\]
for the eigenvalues of \(A(G)\), counted with multiplicity. For a real number \(x\), we write \(x_+:=\max\{x,0\}\) for its positive part. For \(p>0\) we define the \emph{adjacency \(p\)-energy} of \(G\) by
\[
        \mathcal E_p(G):=\sum_{i=1}^n |\lambda_i(G)|^p .
\]
We write \(P_n\) for the path on \(n\) vertices. The companion paper~\cite{LiuTangCore} proves that, for every connected graph
\(G\) on \(n\) vertices and every real \(p\ge2\),
\[
        \mathcal E_p(G)\ge \mathcal E_p(P_n),
\]
with equality for fixed \(p>2\) only when \(G\cong P_n\).  The proof also
establishes a stronger second-order stop-loss comparison for the squared
singular values of connected bipartite graphs.  The purpose of the present
paper is to turn these comparison principles into sharp extremal results for
several related spectral quantities.

The first application concerns positive adjacency \(p\)-energy.  If
\[
        \mathcal E_p^+(G):=\sum_{\lambda_i(G)>0}\lambda_i(G)^p,
\]
then the conjecture of Tang, Liu and Wang~\cite[Conjecture~4.6]{TangLiuWang2026} asks whether
\[
        \mathcal E_p^+(G)\ge \mathcal E_p^+(P_n)
\]
holds for every connected \(n\)-vertex graph \(G\) and every \(p\ge2\).  We verify this sharp path lower bound
in three cases.  It holds for all connected bipartite graphs and all real
\(p\ge2\), for all connected graphs when \(p\ge3\) is an odd integer, and for
all connected graphs when \(p=4\).  The first two statements are short
consequences of the total \(p\)-energy theorem, while the \(p=4\) case uses a
vertex-addition estimate for positive \(p\)-energy.

The second application is to Laplacian-type spectra. For a graph \(G\), let \(D(G)\) denote the diagonal matrix of vertex degrees and let \(A(G)\) denote the adjacency matrix.  The \emph{Laplacian matrix} and the \emph{signless Laplacian matrix} of \(G\) are
\[
        L(G):=D(G)-A(G),
        \qquad
        Q(G):=D(G)+A(G),
\]
respectively.  Both matrices are real symmetric and positive semidefinite. For a real symmetric matrix \(M\), we write $\lambda_1(M)\ge\lambda_2(M)\ge\cdots$ for its eigenvalues in nonincreasing order, repeated with multiplicity. Using the subdivision graph and the unsigned incidence matrix, we convert the bipartite comparison
from the companion paper~\cite{LiuTangCore} into path-minimality statements for \(L(G)\) and
\(Q(G)\).  In particular, for every connected \(G\) on \(n\) vertices and every
\(\alpha\ge1\), we prove
\[
        \sum_{i=1}^n \lambda_i(L(G))^\alpha
        \ge
        \sum_{i=1}^n \lambda_i(L(P_n))^\alpha,
        \qquad
        \sum_{i=1}^n \lambda_i(Q(G))^\alpha
        \ge
        \sum_{i=1}^n \lambda_i(Q(P_n))^\alpha .
\]
More generally, the same method gives comparisons for  Estrada-type sums, resolvent energies,
and thresholded tails.

The third application concerns line graphs.  We use the standard spectral
relation
\[
        \mathcal E_p^+(\mathcal L(G))
        =
        \sum_i\bigl(\lambda_i(Q(G))-2\bigr)_+^p.
\]
To prove the desired
line-graph inequality, we establish the following edge-count second-order
stop-loss comparison:
\[
        \sum_i\bigl(\lambda_i(Q(G))-t\bigr)_+^2
        \ge
        \sum_i\bigl(\lambda_i(Q(P_{m+1}))-t\bigr)_+^2
        \qquad(t\ge2),
\]
for every connected graph \(G\) with \(m\) edges.  Integrating this comparison
over the threshold gives
\[
        \mathcal E_p^+(\mathcal L(G))\ge \mathcal E_p^+(P_m)
        \qquad(p\ge2).
\]
This is the technically most independent part of the paper; its proof uses
rank-one spectral-shift estimates, a deletion-minimal counterexample argument,
and exact finite checks.

The paper is organized as follows.  Section~\ref{sec:inputs} recalls the
comparison inputs from the companion paper and collects the spectral-shift
notation used later.  Section~\ref{sec:applications} proves the applications
to positive \(p\)-energies, Laplacian-type spectra, and line graphs.
Appendix~\ref{app:linegraph-scalar-check} records the exact rational
certificate checks used in the line-graph argument.

\section{Preliminaries and comparison inputs}
\label{sec:inputs}

We first recall the comparison results from the companion paper that will be
used below.  The remainder of the section collects elementary spectral-shift
notation and estimates needed in the line-graph argument.

\begin{theorem}[{\cite[Theorem~1.3]{LiuTangCore}}]
\label{thm:pgeq2}
Let \(G\) be a connected graph on \(n\) vertices.  Then, for every real number
\(p\ge2\),
\[
        \mathcal E_p(G)\ge \mathcal E_p(P_n).
\]
Moreover, for each fixed \(p>2\), equality holds if and only if \(G\cong P_n\).
\end{theorem}

For a bipartite graph \(G\), let \(\mu_i(G)\) denote the squared singular values
of a biadjacency matrix of \(G\).  For \(t\ge0\), define
\[
        \mathsf S_t(G):=\sum_i(\mu_i(G)-t)_+^2 .
\]
The sum may be taken over positive squared singular values only, or with
additional zeroes; this makes no difference for \(t\ge0\).

\begin{definition}[Admissible third-order convex functions, {\cite[Definition~4.2]{LiuTangCore}}]
A function \(F:[0,\infty)\to\mathbb R\) is called
\emph{admissible third-order convex} if
\[
        F\in C^2([0,\infty)),\qquad
        F''\in AC_{\mathrm{loc}}([0,\infty)),
\]
where \(AC_{\mathrm{loc}}([0,\infty))\) means absolute continuity on every compact subinterval of \([0,\infty)\), and
\[
        F(0)=0,\qquad F'(0)\ge0,\qquad F''(0)\ge0,
\]
while the a.e. derivative \(F'''\) satisfies
\[
        F'''(t)\ge0
        \qquad\text{for a.e. }t>0.
\]
\end{definition}

\begin{theorem}[{\cite[Theorem~4.5]{LiuTangCore}}]
\label{thm:bipartite-stoploss}
Let \(G\) be a connected bipartite graph on \(n\) vertices. Then
\[
        \mathsf S_t(G)\ge \mathsf S_t(P_n)
        \qquad(t\ge0).
\]
\end{theorem}

\begin{corollary}[{\cite[Corollary~4.6]{LiuTangCore}}]
\label{cor:connected-third-order-convex}
Let \(G\) be a connected graph on \(n\) vertices.  Then, for every admissible
third-order convex function \(F\),
\[
        \sum_{i=1}^n F\left(\lambda_i(G)^2\right)
        \ge
        \sum_{i=1}^n F\left(\lambda_i(P_n)^2\right).
\]
\end{corollary}

For an interval \([a,b]\subset[0,\infty)\), write
\[
        \mathcal I_{[a,b]}(t):=\int_a^b (u-t)_+\,du.
\]
For a finite interval set \(E\subset[0,\infty)\), write
\[
        J_E(t):=\int_E(y-t)_+\,dy .
\]

\begin{lemma}[{\cite[Lemma~5.1]{LiuTangCore}}]
\label{lem:rank-one-trace-formula}
Let \(M\) be a real symmetric matrix, let \(b\) be a vector of the same
dimension, and set \(M_\theta=M+\theta bb^\top\) for \(0\le\theta\le1\). Then,
for every \(t\ge0\),
\[
        \tr(M_1-tI)_+^2-\tr(M_0-tI)_+^2
        =
        2\int_0^1 b^\top(M_\theta-tI)_+b\,d\theta .
\]
\end{lemma}

For real symmetric matrices \(A\) and \(B\) of the same size, we write \(A\succeq B\) if \(A-B\) is positive semidefinite. Let \(M\succeq0\) and \(b\ne0\), and consider the positive rank-one update
\(M\rightsquigarrow M+bb^\top\).  Let
\[
        \alpha_1\ge\alpha_2\ge\cdots\ge0,
        \qquad
        \beta_1\ge\beta_2\ge\cdots\ge0
\]
be the eigenvalues of \(M\) and \(M+bb^\top\), respectively, with zero padding
if necessary.  Rank-one interlacing gives
\[
        \beta_1\ge \alpha_1\ge \beta_2\ge \alpha_2\ge\cdots\ge0 .
\]
We define the associated spectral-shift interval set by
\[
        E(M,b):=\bigcup_i[\alpha_i,\beta_i],
\]
where zero-length intervals and endpoint overlaps are harmless.  Thus
\(E(M,b)\) records the interlacing gaps created by the positive rank-one
update.

For a finite interval set \(E\subset[0,\infty)\), recall that
\[
        J_E(t)=\int_E (y-t)_+\,dy,
        \qquad t\ge0.
\]
When \(E=E(M,b)\), this quantity is half of the increase of the second-order stop-loss functional:
\[
        \tr(M+bb^\top-tI)_+^2-\tr(M-tI)_+^2
        =
        2J_{E(M,b)}(t).
\]
If \(M_\theta=M+\theta bb^\top\), then combining this identity with Lemma~\ref{lem:rank-one-trace-formula} gives
\[
        J_{E(M,b)}(t)=\int_0^1 b^\top(M_\theta-tI)_+b\,d\theta .
\]

\begin{lemma}[{\cite[Lemma~5.4]{LiuTangCore}}]
\label{lem:spectral-shift-packing}
Let \(E\subset[0,4]\) be a finite disjoint union of intervals, and let
\(\ell=|E|\) be its total length. Then, for every \(t\ge0\),
\[
        J_E(t)\le \mathcal I_{[4-\ell,4]}(t).
\]
\end{lemma}

For the path endpoint increment \(P_{n-1}\rightsquigarrow P_n\), label the
vertices as \(1,\ldots,n\), so that \(P_{n-1}\) is the path
\(1-\cdots-(n-1)\) and the new edge is \((n-1)n\).  Choose the bipartition so
that the new endpoint \(n\) lies on the left side; then the old endpoint
\(n-1\) lies on the right side.  Let \(B_{P_{n-1}}\) be the biadjacency matrix
of \(P_{n-1}\) with this bipartition, and put
\[
        M_0:=B_{P_{n-1}}^\top B_{P_{n-1}}.
\]
Let \(e\) be the standard basis vector corresponding to the right-side vertex
\(n-1\).  Adding the new endpoint appends the row \(e^\top\) to the
biadjacency matrix, and therefore $B_{P_n}^\top B_{P_n}=M_0+ee^\top$. We write 
\[
        E_n^P:=E(M_0,e)
\]
for the corresponding spectral-shift interval set.

\begin{lemma}[{\cite[Lemma~6.3]{LiuTangCore}}]\label{lem:path-moments}
For every \(n\ge2\), let \(E_n^P\) be the path endpoint spectral-shift interval
set for \(P_{n-1}\rightsquigarrow P_n\). Then, for every integer \(m\ge1\),
\[
        \int_{E_n^P} y^{m-1}\,dy
        \le
        \frac1{2m}\binom{2m}{m}.
\]
\end{lemma}

\begin{lemma}[{\cite[Lemma~6.6]{LiuTangCore}}]\label{lem:path-moment-envelope}
Let \(E\subset[0,4]\) be a finite union of intervals. For every \(r>1\) and
every \(0<t<4\),
\[
        J_E(t)
        \le
        \Gamma_r(t)\int_E y^r\,dy,
\]
where
\[
        \Gamma_r(t):=
        \max_{t<y\le4}\frac{y-t}{y^r}
        =
        \begin{cases}
        \dfrac{(r-1)^{r-1}}{r^r t^{r-1}},
        &0<t\le \dfrac{4(r-1)}r,\\[3mm]
        \dfrac{4-t}{4^r},
        &\dfrac{4(r-1)}r\le t<4.
        \end{cases}
\]
For \(r=1\), one has \(J_E(t)\le\int_E y\,dy\) for \(t\ge0\).
\end{lemma}

\begin{lemma}[{\cite[Lemma~4.15]{LiuTangCore}}]\label{lem:cycle-pair-ineq}
Let \(m\ge2\) and \(1\le h\le\lfloor m/2\rfloor\). Then
\[
\cos^2\frac{(h-1)\pi}{m}
+
\cos^2\frac{h\pi}{m}
\ge
\cos^2\frac{(2h-1)\pi}{2m+1}
+
\cos^2\frac{2h\pi}{2m+1}.
\]
\end{lemma}

\section{Applications to positive \texorpdfstring{$p$}{p}-energies, Laplacian-type spectra, and line graphs}
\label{sec:applications}

We now prove the announced applications.  The first two subsections use the
path-minimality and bipartite stop-loss comparisons from the companion paper
rather directly.  The final subsection proves an additional edge-count
comparison for the signless Laplacian, which is then converted into a sharp
positive \(p\)-energy bound for line graphs.

\subsection{Applications to positive \texorpdfstring{$p$}{p}-energies}

Positive and negative $p$-energies appear naturally in several spectral graph-theoretic problems.  For example, Elphick, Tang and Zhang used positive and negative $p$-energies to obtain lower bounds for the chromatic number, fractional chromatic number, quantum chromatic number, orthogonal rank, and projective rank; their results show that non-integer values of $p$ can sometimes give sharper spectral coloring bounds than the classical choices~\cite{ElphickTangZhang2026}.

For a graph $G$, let
\[
        \mathcal E_p^+(G)=\sum_{\lambda_i(G)>0}\lambda_i(G)^p,
        \qquad
        \mathcal E_p^-(G)=\sum_{\lambda_i(G)<0}|\lambda_i(G)|^p
\]
be its \emph{positive} and \emph{negative adjacency $p$-energies}.  When $p=2$, these quantities are the positive and negative square energies of $G$.  The positive and negative square energies had appeared earlier in the literature, but they were first studied systematically by Elphick, Farber, Goldberg and Wocjan, who formulated influential conjectural lower bounds for the sums of squares of positive and negative eigenvalues~\cite{ElphickFarberGoldbergWocjan2016}.  For recent progress on these square-energy problems, see the works of Zhang~\cite{Zhang2024SquareEnergies}, Akbari, Kumar, Mohar and Pragada~\cite{AkbariKumarMoharPragada2025SquareEnergy}, and Akbari, Kumar, Mohar, Pragada and Zhang~\cite{AkbariKumarMoharPragadaZhang2025Refinement}.

Tang, Liu and Wang introduced positive and negative $p$-energies in their study of edge addition and proposed the following positive $p$-energy conjecture~\cite{TangLiuWang2026}.

\begin{conjecture}[{\cite[Conjecture~4.6]{TangLiuWang2026}}]\label{conj:positive-p-energy}
Let $p\ge2$ be a real number, and let $G$ be a connected graph with $n$ vertices.  Then
\[
        \mathcal E_p^+(G)\ge \mathcal E_p^+(P_n).
\]
\end{conjecture}

Conjecture~\ref{conj:positive-p-energy}, together with the related problems
for positive and negative \(p\)-energies, has attracted recent attention.
Akbari, Kumar, Mohar and Pragada proved a general superadditivity theorem for
the positive and negative \(p\)-energies of Hermitian block matrices for every
\(p\ge1\)~\cite{AKMP2025}.  They also proposed the corresponding negative
\(p\)-energy conjecture, asserting that
\[
        \mathcal E_p^-(G)\ge \mathcal E_p^-(K_n)
\]
for every connected graph \(G\) on \(n\) vertices and every \(p\ge2\), and
proved this conjecture for \(p\ge4\)~\cite{AKMP2025}. In the positive direction, they obtained a linear lower bound for
\(\mathcal E_p^+(G)\) when \(p\ge4\)~\cite[Theorem~2]{AKMP2025}. Recently, Chen, Wang and Zhang proved the negative
\(p\)-energy conjecture for every integer \(p\ge3\) and obtained a linear lower
bound for the positive \(3\)-energy~\cite{ChenWangZhang2026}.

For clarity, none of the preceding positive-energy results gives the exact
path lower bound asserted in Conjecture~\ref{conj:positive-p-energy}.  The
positive-energy lower bound of Akbari, Kumar, Mohar and Pragada has the correct
linear order of magnitude for \(p\ge4\), but it is not the sharp quantity
\(\mathcal E_p^+(P_n)\).  Similarly, the positive-energy result of Chen, Wang
and Zhang gives a linear lower bound for \(\mathcal E_3^+(G)\), rather than the
exact path-minimality inequality $\mathcal E_3^+(G)\ge \mathcal E_3^+(P_n)$. Thus the exact positive \(p\)-energy path bound remains separate from these
known linear lower bounds, even in the cases \(p=3\) and \(p\ge4\).

Our main theorem settles Conjecture~\ref{conj:positive-p-energy} for connected bipartite graphs.

\begin{theorem}\label{thm:positive-bipartite}
Let $G$ be a connected bipartite graph on $n$ vertices.  Then, for every real number $p\ge2$,
\[
        \mathcal E_p^+(G)\ge \mathcal E_p^+(P_n).
\]
\end{theorem}

\begin{proof}
Since $G$ is bipartite, its nonzero adjacency eigenvalues occur in pairs $\pm\lambda$.  Hence
\[
        \mathcal E_p^+(G)=\mathcal E_p^-(G)=\frac12\mathcal E_p(G).
\]
The path $P_n$ is also bipartite, so $\mathcal E_p^+(P_n)=\frac12\mathcal E_p(P_n)$.  By Theorem~\ref{thm:pgeq2}, we have $\mathcal E_p(G)\ge\mathcal E_p(P_n)$. Therefore
\[
        \mathcal E_p^+(G)
        =
        \frac12\mathcal E_p(G)
        \ge
        \frac12\mathcal E_p(P_n)
        =
        \mathcal E_p^+(P_n).
\qedhere\]
\end{proof}

The same theorem also verifies Conjecture~\ref{conj:positive-p-energy} for every odd integer \(p\ge3\).  In particular, when \(p=3\), it gives the sharp path lower bound predicted by the conjecture,
\[
        \mathcal E_3^+(G)\ge \mathcal E_3^+(P_n),
\]
going beyond the linear lower bound for the positive \(3\)-energy obtained by Chen, Wang and Zhang~\cite[Theorem~1.8]{ChenWangZhang2026}.

\begin{theorem}\label{thm:positive-odd-integer}
Let $p\ge3$ be an odd integer.  Then, for every connected graph $G$ on $n$ vertices,
\[
        \mathcal E_p^+(G)\ge \mathcal E_p^+(P_n).
\]
\end{theorem}

\begin{proof}
Let $A=A(G)$.  Since $p$ is odd,
\[
        \tr A^p
        =
        \sum_i\lambda_i(G)^p
        =
        \mathcal E_p^+(G)-\mathcal E_p^-(G).
\]
On the other hand, $\tr A^p$ counts closed walks of length $p$ in $G$.  Therefore $\tr A^p\ge0$.  Consequently,
\[
        2\mathcal E_p^+(G)
        =
        \mathcal E_p(G)+\tr A^p
        \ge
        \mathcal E_p(G).
\]
By Theorem~\ref{thm:pgeq2}, we know that $\mathcal E_p(G)\ge \mathcal E_p(P_n)$. Since $P_n$ is bipartite, $\mathcal E_p(P_n)=2\mathcal E_p^+(P_n)$.  Hence
\[
        2\mathcal E_p^+(G)
        \ge
        \mathcal E_p(P_n)
        =
        2\mathcal E_p^+(P_n),
\]
and the result follows.
\end{proof}

For \(p\ge1\), the \emph{Schatten \(p\)-norm} of a matrix \(M\) is
\[
        \norm{M}_{S_p}:=
        \left(\sum_j s_j(M)^p\right)^{1/p},
\]
where \(s_j(M)\) are the singular values of \(M\).  Since \(A(G)\) is real symmetric, its singular values are \(|\lambda_1(G)|,\ldots,|\lambda_n(G)|\). Thus, for \(p\ge1\), we have $\mathcal E_p(G)=\norm{A(G)}_{S_p}^p$.

We next record a separate argument for the first non-bipartite even exponent. The proof is based on a vertex-addition estimate for positive \(p\)-energy.

\begin{lemma}\label{lem:positive-vertex-gain}
Let \(H\) be a graph, and let \(G\) be obtained from \(H\) by adding one new
vertex \(v\) of degree \(d\ge1\).  Then, for every real number \(p\ge3\),
\[
        \mathcal E_p^+(G)-\mathcal E_p^+(H)\ge d^{p/2}.
\]
\end{lemma}

\begin{proof}
Order the vertices of \(G\) so that the new vertex \(v\) is last, and let \(b\) be the \(0\)-\(1\) column vector whose support is \(N_G(v)\subseteq V(H)\). Then
\[
        A_0:=A(H)\oplus[0]
        =
        \begin{pmatrix}
        A(H)&0\\
        0&0
        \end{pmatrix},
        \qquad
        K:=A(G)-A_0
        =
        \begin{pmatrix}
        0&b\\
        b^{\top}&0
        \end{pmatrix}.
\]
Thus \(K\) is the adjacency matrix of a star with center \(v\) and \(d\) leaves. Its nonzero eigenvalues are \(\pm\sqrt d\), and hence $\norm{K}_{S_3}^3=2d^{3/2}$.

Moreover,
\[
        A_0+K
        =
        \begin{pmatrix}
        A(H)&b\\
        b^{\top}&0
        \end{pmatrix}
        =
        A(G).
\]
We also claim that \(A_0-K\) is orthogonally similar to \(A_0+K\).  Indeed, let
$
        D=
        \begin{pmatrix}
        I&0\\
        0&-1
        \end{pmatrix}
$, where \(I\) is the identity matrix on the vertex set of \(H\).  Then \(D\) is
orthogonal and \(D^{-1}=D\).  A direct multiplication gives
\[
        D(A_0+K)D
        =
        \begin{pmatrix}
        A(H)&-b\\
        -b^{\top}&0
        \end{pmatrix}
        =
        A_0-K.
\]
Thus \(A_0-K\) is orthogonally similar to \(A_0+K=A(G)\).  In particular,
\[
        \norm{A_0-K}_{S_3}=\norm{A_0+K}_{S_3}=\norm{A(G)}_{S_3}.
\]

By the Clarkson--McCarthy inequality for Schatten norms~\cite{McCarthy1967},
\[
        \norm{X+Y}_{S_r}^r+\norm{X-Y}_{S_r}^r
        \ge
        2\norm{X}_{S_r}^r+2\norm{Y}_{S_r}^r
        \qquad(r\ge2).
\]
Taking \(r=3\), \(X=A_0\), and \(Y=K\), we obtain
\[
        2\mathcal E_3(G)
        =
        \norm{A_0+K}_{S_3}^3+\norm{A_0-K}_{S_3}^3
        \ge
        2\mathcal E_3(H)+4d^{3/2}.
\]
Therefore
\[
        \mathcal E_3(G)-\mathcal E_3(H)\ge2d^{3/2}.
\]

We now pass from total \(3\)-energy to positive \(3\)-energy.  Since
\(\tr A(J)^3\) counts the closed walks of length \(3\) in a graph \(J\), and
since every closed walk of length \(3\) in \(H\) is also such a walk in \(G\),
we have
\[
        \tr A(G)^3-\tr A(H)^3\ge0.
\]
Since, for every graph \(J\), $\mathcal E_3^+(J) = \frac12\left(\mathcal E_3(J)+\tr A(J)^3\right)$, we get
\begin{equation}\label{eq:positive-three-vertex-gain}
        \mathcal E_3^+(G)-\mathcal E_3^+(H)\ge d^{3/2}.
\end{equation}

It remains to upgrade this estimate from \(p=3\) to every \(p\ge3\).  Let $\lambda_1\ge\cdots\ge\lambda_{m+1}$ be the eigenvalues of \(G\), and let $\theta_1\ge\cdots\ge\theta_m$ be the eigenvalues of \(H\).  By Cauchy interlacing,
\[
        \lambda_i\ge\theta_i\ge\lambda_{i+1}
        \qquad(1\le i\le m).
\]
Since \(G\) has at least one edge, its adjacency matrix is nonzero and has
trace zero; hence \(\lambda_{m+1}<0\).  In particular,
\((\lambda_{m+1})_+=0\).  Therefore, for every \(p>0\),
\[
        \mathcal E_p^+(G)-\mathcal E_p^+(H)
        =
        \sum_{i=1}^m\left((\lambda_i)_+^p-(\theta_i)_+^p\right)  =
        \sum_{i=1}^m
        \int_{(\theta_i)_+}^{(\lambda_i)_+} p t^{p-1}\,dt .
\]
The intervals
\[
        I_i=\bigl[(\theta_i)_+,(\lambda_i)_+\bigr]
        \qquad(1\le i\le m)
\]
are pairwise disjoint up to endpoints, again by interlacing.  Put
\[
        E=\bigcup_{i=1}^m I_i .
\]
Then
\[
        \mathcal E_p^+(G)-\mathcal E_p^+(H)
        =
        \int_E p t^{p-1}\,dt .
\]
In particular, \eqref{eq:positive-three-vertex-gain} says that
\begin{equation}\label{eq:inte3t2dtged32}
        \int_E 3t^2\,dt\ge d^{3/2}.
\end{equation}

Set \(a=\sqrt d\), and let \(\widetilde E\) be the image of \(E\) under the change of variables \(y=t^3\).  Since \(dy=3t^2\,dt\), \eqref{eq:inte3t2dtged32} means that \(|\widetilde E|\ge a^3\). For \(p\ge3\), the function $y\mapsto \frac p3 y^{p/3-1}$ is nonnegative and nondecreasing on \([0,\infty)\).  Hence, among measurable subsets of \([0,\infty)\) of measure at least \(a^3\), the integral of this function is minimized by the initial interval \([0,a^3]\).  Therefore
\[
        \mathcal E_p^+(G)-\mathcal E_p^+(H)
        =
        \int_{\widetilde E}\frac p3 y^{p/3-1}\,dy  
        \ge
        \int_0^{a^3}\frac p3 y^{p/3-1}\,dy
        =
        a^p
        =
        d^{p/2}.
\]
This proves the lemma.
\end{proof}

The next elementary estimate is the path-side deficit needed in the induction for \(p=4\).

\begin{lemma}\label{lem:path-deficit-positive-four}
Let \(q\ge1\), let \(n_1,\ldots,n_q\ge1\), and put $N=1+n_1+\cdots+n_q$. Then
\[
        \mathcal E_4^+(P_N)-\sum_{i=1}^q\mathcal E_4^+(P_{n_i})
        <
        (q+1)^2 .
\]
\end{lemma}

\begin{proof}
For \(m\ge2\), the path \(P_m\) is bipartite and
\[
        \mathcal E_4(P_m)=\tr A(P_m)^4=6m-10.
\]
Hence
\[
        \mathcal E_4^+(P_m)=\frac12\mathcal E_4(P_m)=3m-5
        \qquad(m\ge2),
\]
while \(\mathcal E_4^+(P_1)=0\).

Let \(s\) be the number of indices \(i\) for which \(n_i=1\).  Since
\(N=1+\sum_i n_i\), we have
\[
\begin{aligned}
        \mathcal E_4^+(P_N)-\sum_{i=1}^q\mathcal E_4^+(P_{n_i})
        &=
        (3N-5)
        -
        \sum_{i:n_i\ge2}(3n_i-5)   \\
        &=
        5q-2-2s
        \le
        5q-2
        <(q+1)^2.
\end{aligned}
\]
This proves the lemma.
\end{proof}

We can now settle Conjecture~\ref{conj:positive-p-energy} at the first
non-bipartite even exponent.  For \(n\ge2\), this gives the exact path lower
bound
\[
        \mathcal E_4^+(G)\ge \mathcal E_4^+(P_n)=3n-5,
\]
and therefore strengthens the \(p=4\) case of the linear lower bound of
Akbari, Kumar, Mohar and Pragada~\cite[Theorem~2]{AKMP2025}.

\begin{theorem}\label{thm:positive-four-energy}
Let \(G\) be a connected graph on \(n\) vertices.  Then
\[
        \mathcal E_4^+(G)\ge \mathcal E_4^+(P_n).
\]
\end{theorem}

\begin{proof}
We argue by induction on \(n\).  The case \(n=1\) is immediate.

If \(G\) is a tree, then \(G\) is bipartite, and the result follows from Theorem~\ref{thm:positive-bipartite}.

Assume now that \(G\) is not a tree.  Choose a vertex \(v\) lying on a cycle.
Let
\[
        d=d_G(v),
        \qquad
        G-v=G_1\sqcup\cdots\sqcup G_q,
        \qquad
        n_i=|V(G_i)|.
\]
Every component of \(G-v\) contains at least one neighbor of \(v\), because
\(G\) is connected.  Moreover, two neighbors of \(v\) on the chosen cycle
remain in the same component after deleting \(v\).  Therefore \(d\ge q+1\).

By Lemma~\ref{lem:positive-vertex-gain}, applied with \(p=4\),
\[
        \mathcal E_4^+(G)-\mathcal E_4^+(G-v)\ge d^2\ge(q+1)^2.
\]
By additivity over connected components and by the induction hypothesis,
\[
        \mathcal E_4^+(G-v)
        =
        \sum_{i=1}^q\mathcal E_4^+(G_i)
        \ge
        \sum_{i=1}^q\mathcal E_4^+(P_{n_i}).
\]
Consequently,
\[
\begin{aligned}
        \mathcal E_4^+(G)
        &=
        \mathcal E_4^+(G-v)
        +
        \bigl(\mathcal E_4^+(G)-\mathcal E_4^+(G-v)\bigr)\\
        &\ge
        \sum_{i=1}^q\mathcal E_4^+(P_{n_i})+(q+1)^2.
\end{aligned}
\]
Since \(n=1+\sum_i n_i\), Lemma~\ref{lem:path-deficit-positive-four} gives
\[
        \sum_{i=1}^q\mathcal E_4^+(P_{n_i})+(q+1)^2
        >
        \mathcal E_4^+(P_n).
\]
Therefore $\mathcal E_4^+(G)>\mathcal E_4^+(P_n)$. This completes the proof.
\end{proof}

\subsection{Laplacian and signless Laplacian path-minimality}

We next turn to Laplacian and signless Laplacian spectra. Power sums of Laplacian and signless Laplacian eigenvalues have been studied
in their own right; see, for instance, Akbari, Ghorbani, Koolen and
Oboudi~\cite{AkbariGhorbaniKoolenOboudi2010} and Kaya and
Maden~\cite{KayaMaden2018}.  We shall derive path-minimality results for these
spectra from the bipartite comparison theorem.

We shall use the following standard construction.  For a graph \(G\), its \emph{subdivision graph} \(S(G)\) is the bipartite graph with vertex set $V(S(G))=V(G)\sqcup E(G)$, where a vertex \(v\in V(G)\) is adjacent in \(S(G)\) to an edge \(e\in E(G)\) exactly when \(v\) is an endpoint of \(e\) in \(G\). Thus each edge \(uv\in E(G)\) is replaced in \(S(G)\) by a path \(u-e-v\) of length two. In particular, $S(P_n)=P_{2n-1}$. The corresponding \emph{unsigned incidence matrix} of \(G\) is the
\(|V(G)|\times |E(G)|\) matrix \(N_G\) defined by
\[
        (N_G)_{v,e}
        =
        \begin{cases}
        1, & \text{if }v\text{ is an endpoint of }e,\\
        0, & \text{otherwise}.
        \end{cases}
\]
With respect to the bipartition \(V(G)\sqcup E(G)\), the biadjacency matrix of
\(S(G)\) is precisely \(N_G\).

We first record a stronger comparison for admissible third-order convex test
functions.  The power-sum inequalities will then follow as a separate
corollary.

\begin{theorem}\label{thm:lap-signless-master}
Let \(G\) be a connected graph on \(n\) vertices.  If
\(F:[0,\infty)\to\mathbb R\) is admissible third-order convex, then
\[
        \sum_{i=1}^n F(\lambda_i(L(G)))
        \ge
        \sum_{i=1}^n F(\lambda_i(L(P_n)))
\quad
\text{and}
\quad
        \sum_{i=1}^n F(\lambda_i(Q(G)))
        \ge
        \sum_{i=1}^n F(\lambda_i(Q(P_n))).
\]
\end{theorem}

\begin{proof}
We first prove the result for trees.  Let \(T\) be a tree on \(n\) vertices, and let \(S(T)\) be its subdivision graph.  We regard \(S(T)\) as a bipartite graph with bipartition \(V(T)\sqcup E(T)\).  With this bipartition, the biadjacency matrix of \(S(T)\) is the unsigned incidence matrix \(N_T\) of \(T\), whose rows are indexed by \(V(T)\) and whose columns are indexed by \(E(T)\). Hence $N_TN_T^{\top}=Q(T)$. Indeed, the diagonal entries of \(N_TN_T^{\top}\) are the vertex degrees of \(T\), while its off-diagonal \((u,v)\)-entry is \(1\) exactly when
\(uv\in E(T)\), and is \(0\) otherwise.

Since \(T\) is bipartite, \(Q(T)\) is similar to \(L(T)\).  Indeed, if \(R\) is the diagonal matrix whose diagonal entries are \(1\) on one side of a bipartition of \(T\) and \(-1\) on the other side, then $RQ(T)R=L(T)$. Thus \(L(T)\) and \(Q(T)\) have the same eigenvalues.

The nonzero squared singular values of the biadjacency matrix of \(S(T)\) are
precisely the nonzero eigenvalues of \(N_TN_T^{\top}=Q(T)\).  Since
\(F(0)=0\), possible additional zero eigenvalues do not affect any of the sums
below.  Therefore the squared-singular-value sum for \(S(T)\) agrees with the
\(F\)-sum over the Laplacian, or equivalently the signless Laplacian,
eigenvalues of \(T\).

Applying Corollary~\ref{cor:connected-third-order-convex} to the connected bipartite graph \(S(T)\), and using \(S(P_n)=P_{2n-1}\), gives
\[
        \sum_i F(\mu_i(S(T)))
        \ge
        \sum_i F(\mu_i(S(P_n))),
\]
where \(\mu_i\) denotes the squared singular values used in the definition of
\(\mathsf S_t\).  By the spectral identification above, this is exactly
\[
        \sum_{i=1}^n F(\lambda_i(L(T)))
        =
        \sum_{i=1}^n F(\lambda_i(Q(T)))
        \ge
        \sum_{i=1}^n F(\lambda_i(L(P_n)))
        =
        \sum_{i=1}^n F(\lambda_i(Q(P_n))).
\]
Thus the desired inequalities hold when \(G\) is a tree.

Now let \(G\) be an arbitrary connected graph, and choose a spanning tree
\(T\) of \(G\).  For an edge \(e=uv\), let \(e_u,e_v\) be the corresponding
standard basis vectors.  The contribution of \(e\) to the Laplacian is
\((e_u-e_v)(e_u-e_v)^{\top}\), while its contribution to the signless
Laplacian is \((e_u+e_v)(e_u+e_v)^{\top}\).  Hence
\[
        L(G)-L(T)\succeq0,
        \qquad
        Q(G)-Q(T)\succeq0 .
\]
Therefore,
\[
        \lambda_i(L(G))\ge \lambda_i(L(T)),
        \qquad
        \lambda_i(Q(G))\ge \lambda_i(Q(T))
        \qquad(1\le i\le n).
\]

It remains only to note that every admissible third-order convex function is nondecreasing on \([0,\infty)\).  Indeed, \(F'''\ge0\) almost everywhere and \(F''\in AC_{\mathrm{loc}}([0,\infty))\) imply
\[
        F''(x)=F''(0)+\int_0^x F'''(s)\,ds\ge0
        \qquad(x\ge0),
\]
and hence
\[
        F'(x)=F'(0)+\int_0^x F''(s)\,ds\ge0
        \qquad(x\ge0).
\]
Therefore the inequalities already proved for \(T\) imply the corresponding
inequalities for \(G\).
\end{proof}

\begin{corollary}\label{cor:lap-signless-powers}
Let \(G\) be a connected graph on \(n\) vertices.  Then, for every real number
\(\alpha\ge1\),
\[
        \sum_{i=1}^n \lambda_i(L(G))^\alpha
        \ge
        \sum_{i=1}^n \lambda_i(L(P_n))^\alpha
\quad
\text{and}
\quad
        \sum_{i=1}^n \lambda_i(Q(G))^\alpha
        \ge
        \sum_{i=1}^n \lambda_i(Q(P_n))^\alpha .
\]
\end{corollary}

\begin{proof}
We first prove the result for trees.  Let \(T\) be a tree on \(n\) vertices.
Recall from the proof of Theorem~\ref{thm:lap-signless-master} that the
nonzero squared singular values of the biadjacency matrix of the subdivision
graph \(S(T)\) are precisely the nonzero eigenvalues of \(Q(T)\).  Since \(T\)
is bipartite, \(Q(T)\) and \(L(T)\) are cospectral.  Thus, for every
\(\alpha\ge1\),
\[
        \mathcal E_{2\alpha}(S(T))
        =
        2\sum_{i=1}^n \lambda_i(L(T))^\alpha
        =
        2\sum_{i=1}^n \lambda_i(Q(T))^\alpha.
\]

Put \(p=2\alpha\).  Since \(p\ge2\), Theorem~\ref{thm:pgeq2}, applied to the
connected graph \(S(T)\), gives
\[
        \mathcal E_{2\alpha}(S(T))
        \ge
        \mathcal E_{2\alpha}(P_{2n-1}).
\]
Since \(S(P_n)=P_{2n-1}\), the same spectral identification applied to \(P_n\)
gives
\[
        \mathcal E_{2\alpha}(P_{2n-1})
        =
        \mathcal E_{2\alpha}(S(P_n))
        =
        2\sum_{i=1}^n \lambda_i(L(P_n))^\alpha
        =
        2\sum_{i=1}^n \lambda_i(Q(P_n))^\alpha.
\]
Combining the last two displays proves both the Laplacian and signless
Laplacian power-sum inequalities for trees.

Now let \(G\) be an arbitrary connected graph, and choose a spanning tree \(T\) of \(G\).  As in the proof of Theorem~\ref{thm:lap-signless-master}, we have
\[
        \lambda_i(L(G))\ge \lambda_i(L(T)),
        \qquad
        \lambda_i(Q(G))\ge \lambda_i(Q(T))
        \qquad(1\le i\le n).
\]
Since \(x\mapsto x^\alpha\) is increasing on \([0,\infty)\), the tree case
implies the desired inequalities for \(G\).
\end{proof}

The same argument gives the following stronger tree-level second-order stop-loss comparison.

\begin{remark}\label{rem:laplacian-stoploss-form}
Let \(T\) be a tree on \(n\) vertices.  Then, for every \(t\ge0\),
\[
        \sum_{i=1}^n \bigl(\lambda_i(L(T))-t\bigr)_+^2
        \ge
        \sum_{i=1}^n \bigl(\lambda_i(L(P_n))-t\bigr)_+^2,
\]
and
\[
        \sum_{i=1}^n \bigl(\lambda_i(Q(T))-t\bigr)_+^2
        \ge
        \sum_{i=1}^n \bigl(\lambda_i(Q(P_n))-t\bigr)_+^2 .
\]
Indeed, applying Theorem~\ref{thm:bipartite-stoploss} to the subdivision graph \(S(T)\) gives $\mathsf S_t(S(T))\ge \mathsf S_t(P_{2n-1})$. Since \(S(P_n)=P_{2n-1}\), this is the same as $\mathsf S_t(S(T))\ge \mathsf S_t(S(P_n))$. As explained in the proof of Theorem~\ref{thm:lap-signless-master}, the
nonzero squared singular values of the biadjacency matrix of \(S(T)\) are
precisely the nonzero eigenvalues of \(Q(T)\), and hence also of \(L(T)\),
because \(T\) is bipartite.  The possible zero eigenvalues do not contribute,
since \((0-t)_+^2=0\) for \(t\ge0\).  The same identification applies to
\(P_n\), and the two displayed inequalities follow.
\end{remark}

Theorem~\ref{thm:lap-signless-master} and Corollary~\ref{cor:lap-signless-powers} yield unified proofs of several path-minimality statements for Laplacian-type spectral indices. We first
record the exponential, or Estrada-type, consequences. For example, Ili\'c and Zhou proved that \(P_n\) minimizes the Laplacian Estrada index among trees~\cite{IlicZhou2010}; the following result gives simultaneous Laplacian and signless Laplacian versions for every connected graph and every positive
parameter.

\begin{corollary}\label{cor:lap-signless-estrada}
Let \(G\) be a connected graph on \(n\) vertices.  Then, for every
\(\theta>0\),
\[
        \sum_{i=1}^n e^{\theta\lambda_i(L(G))}
        \ge
        \sum_{i=1}^n e^{\theta\lambda_i(L(P_n))}
\qquad
\text{and}
\qquad
        \sum_{i=1}^n e^{\theta\lambda_i(Q(G))}
        \ge
        \sum_{i=1}^n e^{\theta\lambda_i(Q(P_n))}.
\]
\end{corollary}

\begin{proof}
Let $F_\theta(x):=e^{\theta x}-1$. Then \(F_\theta\) is admissible third-order convex: it satisfies
\[
        F_\theta(0)=0,
        \qquad
        F_\theta'(0)=\theta>0,
        \qquad
        F_\theta''(0)=\theta^2>0,
\]
and $F_\theta'''(x)=\theta^3e^{\theta x}\ge0$ for every $x\ge0$. Applying Theorem~\ref{thm:lap-signless-master} to \(F_\theta\), and then adding the common constant \(n\) to both sides, gives the two asserted inequalities.
\end{proof}

We next record the corresponding comparison for the Laplacian and signless Laplacian resolvent energies considered by Sun and Das~\cite{SunDas2019}. Here the proof uses the power-sum comparison in Corollary~\ref{cor:lap-signless-powers}.

\begin{corollary}\label{cor:lap-signless-resolvent}
Let \(G\) be a connected graph on \(n\) vertices.  Define
\[
        RL(G):=\sum_{i=1}^n\frac1{n+1-\lambda_i(L(G))},
        \qquad
        RQ(G):=\sum_{i=1}^n\frac1{2n-1-\lambda_i(Q(G))}.
\]
Then
\[
        RL(G)\ge RL(P_n),
        \qquad
        RQ(G)\ge RQ(P_n).
\]
\end{corollary}

\begin{proof}
We use Corollary~\ref{cor:lap-signless-powers}. The standard spectral bounds
\[
        \lambda_i(L(G))\le n,
        \qquad
        \lambda_i(Q(G))\le 2n-2
\]
imply
\[
        0\le \frac{\lambda_i(L(G))}{n+1}<1,
        \qquad
        0\le \frac{\lambda_i(Q(G))}{2n-1}<1.
\]
Hence the following geometric-series expansions converge absolutely:
\[
\begin{aligned}
        RL(G)
        &=\sum_{i=1}^n\frac1{n+1-\lambda_i(L(G))}  \\
        &=\frac{n}{n+1}
          +\sum_{k\ge1}\frac1{(n+1)^{k+1}}
            \sum_{i=1}^n \lambda_i(L(G))^k .
\end{aligned}
\]
The same expansion holds for \(P_n\).  Comparing term by term, using
Corollary~\ref{cor:lap-signless-powers} with \(\alpha=k\), gives $RL(G)\ge RL(P_n)$. Similarly,
\[
        RQ(G)
        =
        \frac{n}{2n-1}
        +\sum_{k\ge1}\frac1{(2n-1)^{k+1}}
          \sum_{i=1}^n \lambda_i(Q(G))^k,
\]
and the same termwise comparison gives $RQ(G)\ge RQ(P_n)$.
\end{proof}

The preceding stop-loss comparison is closely related in spirit to the Radenkovi\'c--Gutman conjecture on Laplacian energy~\cite{RadenkovicGutman2007}, but it does not settle that conjecture.

\begin{remark}\label{rem:rg-one-order-short}
Recall that the Laplacian energy of a graph \(G\) is
\[
        LE(G)=\sum_{i=1}^n\left|\lambda_i(L(G))-\bar d\right|,
        \qquad
        \bar d=\frac{2|E(G)|}{n}.
\]
Radenkovi\'c and Gutman conjectured, in particular, that \(P_n\) minimizes \(LE\) among all trees on \(n\) vertices~\cite{RadenkovicGutman2007}; see also the recent partial results in~\cite{GanieRatherPirzada2022}.  For trees on \(n\) vertices, the average degree is fixed: $\bar d=2(n-1)/n$. Moreover, $\sum_{i=1}^n\lambda_i(L(T))=n\bar d$. Hence the positive and negative deviations of the numbers \(\lambda_i(L(T))-\bar d\) have the same total mass, and therefore
\[
        LE(T)
        =
        2\sum_{i=1}^n\bigl(\lambda_i(L(T))-\bar d\bigr)_+.
\]
Thus the conjectured path-minimality of Laplacian energy would require the
first-order stop-loss inequality
\[
        \sum_{i=1}^n\bigl(\lambda_i(L(T))-\bar d\bigr)_+
        \ge
        \sum_{i=1}^n\bigl(\lambda_i(L(P_n))-\bar d\bigr)_+ .
\]
Our result in Remark~\ref{rem:laplacian-stoploss-form} gives the second-order comparison
\[
        \sum_{i=1}^n\bigl(\lambda_i(L(T))-t\bigr)_+^2
        \ge
        \sum_{i=1}^n\bigl(\lambda_i(L(P_n))-t\bigr)_+^2
        \qquad(t\ge0),
\]
which in general does not imply the corresponding first-order comparison.
Thus we do not claim a resolution of the Radenkovi\'c--Gutman conjecture here.
\end{remark}

The preceding theorem and the tree-level stop-loss comparison also give
thresholded Laplacian and signless Laplacian tail inequalities.

\begin{corollary}\label{cor:thresholded-lap-signless}
Let \(G\) be a connected graph on \(n\) vertices.  Then, for every \(a\ge0\)
and every real number \(p\ge2\),
\[
        \sum_{i=1}^n \bigl(\lambda_i(L(G))-a\bigr)_+^p
        \ge
        \sum_{i=1}^n \bigl(\lambda_i(L(P_n))-a\bigr)_+^p,
\]
and
\[
        \sum_{i=1}^n \bigl(\lambda_i(Q(G))-a\bigr)_+^p
        \ge
        \sum_{i=1}^n \bigl(\lambda_i(Q(P_n))-a\bigr)_+^p .
\]
\end{corollary}

\begin{proof}
We first consider the endpoint \(p=2\).  Choose a spanning tree \(T\) of
\(G\).  As in the proof of Theorem~\ref{thm:lap-signless-master}, we know that
\[
        \lambda_i(L(G))\ge \lambda_i(L(T)),
        \qquad
        \lambda_i(Q(G))\ge \lambda_i(Q(T))
        \qquad(1\le i\le n).
\]
Since \(x\mapsto (x-a)_+^2\) is increasing on \([0,\infty)\), and since
Remark~\ref{rem:laplacian-stoploss-form} gives the tree-level second-order
stop-loss comparison, we obtain
\[
        \sum_{i=1}^n \bigl(\lambda_i(L(G))-a\bigr)_+^2
        \ge
        \sum_{i=1}^n \bigl(\lambda_i(L(P_n))-a\bigr)_+^2,
\]
and
\[
        \sum_{i=1}^n \bigl(\lambda_i(Q(G))-a\bigr)_+^2
        \ge
        \sum_{i=1}^n \bigl(\lambda_i(Q(P_n))-a\bigr)_+^2.
\]

Now assume \(p>2\). Define $F_{a,p}(x):=(x-a)_+^p$ for every $x\ge0$. It is easy to check that \(F_{a,p}\) is admissible third-order convex. Applying Theorem~\ref{thm:lap-signless-master} to \(F_{a,p}\) gives the two desired inequalities for \(p>2\).
\end{proof}

\subsection{Edge-count signless Laplacian comparison and line graphs}

The thresholded signless Laplacian comparison above is closely related to
positive \(p\)-energies of line graphs.  Recall that the \emph{line graph}
\(\mathcal L(G)\) has vertex set \(E(G)\), with two vertices adjacent in
\(\mathcal L(G)\) precisely when the corresponding edges of \(G\) share an
endpoint.  We write \(\mathcal L(G)\) in order to avoid confusion with the
Laplacian matrix \(L(G)\).

Let \(N_G\) be the unsigned incidence matrix of \(G\).  Then
\[
        N_GN_G^\top=Q(G),
        \qquad
        N_G^\top N_G=2I+A(\mathcal L(G)).
\]
The two matrices \(N_GN_G^\top\) and \(N_G^\top N_G\) have the same nonzero
eigenvalues.  Hence the positive adjacency eigenvalues of \(\mathcal L(G)\)
are exactly the positive numbers among \(\lambda_i(Q(G))-2\).  Equivalently,
for every \(p>0\),
\[
        \mathcal E_p^+(\mathcal L(G))
        =
        \sum_i\bigl(\lambda_i(Q(G))-2\bigr)_+^p .
\]
Thus, in order to obtain the sharp comparison with the path on \(|E(G)|\)
vertices, it is natural to prove an edge-count version of the signless
Laplacian stop-loss comparison above the threshold \(2\).  The deletion method
used in the proof of Theorem~\ref{thm:bipartite-stoploss} gives such a
comparison.

For a graph \(G\) and \(t\ge2\), put
\[
        \Psi_t(G):=\tr\bigl(Q(G)-tI\bigr)_+^2
        =\sum_i\bigl(\lambda_i(Q(G))-t\bigr)_+^2 .
\]
We shall also use the spectral-shift interval notation recalled in Section~\ref{sec:inputs}. Thus, for a rank-one update \(M\rightsquigarrow M+bb^\top\) with \(M\succeq0\), the interval set \(E(M,b)\), formed from the interlacing eigenvalue gaps, satisfies
\[
        \tr(M+bb^\top-tI)_+^2-\tr(M-tI)_+^2
        =2J_{E(M,b)}(t),
        \quad
        \text{where }
        J_E(t)=\int_E(y-t)_+\,dy .
\]

The first local estimate controls the endpoint update for paths.

\begin{lemma}\label{lem:lg-path-endpoint-envelope}
Let \(m\ge1\), and let \(F_m\) be the spectral-shift interval set for the rank-one update
\[
        Q(P_m)\oplus[0]
        \rightsquigarrow
        Q(P_{m+1}),
\]
obtained by adding one pendant edge to the path \(P_m\).  Then
\(F_m\subset[0,4]\), the total length of \(F_m\) is \(2\), and for every
integer \(r\ge1\),
\begin{equation}\label{eq:lg-path-endpoint-moment}
        \int_{F_m} y^r\,dy
        \le
        \frac1{r+1}\binom{2r+2}{r+1}.
\end{equation}
Moreover, with
\[
        A_*(t):=2140-7t^4,
        \qquad
        B_*(t):=4414-4t^4,
        \qquad
        K_*(t):=5(t^3+5t^2+25t+125),
\]
\[
        L_{\rm av}(t):=\frac12\left(\frac72-t\right)^2,
        \qquad
        L_{5,*}(t):=
        \frac{\frac12A_*(t)^2B_*(t)^2+1000A_*(t)^3}{B_*(t)^3K_*(t)},
\]
one has
\[
        J_{F_m}(t)\le L_{\rm av}(t)\qquad(2\le t\le3),
\]
and
\[
        J_{F_m}(t)\le L_{5,*}(t)\qquad(3\le t\le4).
\]
\end{lemma}

\begin{proof}
The update \(Q(P_m)\oplus[0]\rightsquigarrow Q(P_{m+1})\) is obtained by
adding one pendant edge.  If the new edge joins the old endpoint \(u\) to the
new vertex \(w\), then
\[
        Q(P_{m+1})-\bigl(Q(P_m)\oplus[0]\bigr)
        =
        (e_u+e_w)(e_u+e_w)^\top .
\]
Thus the update has the form \(M\rightsquigarrow M+cc^\top\), with
\(\|c\|_2^2=2\).  The total length of the corresponding spectral-shift interval
set is therefore
\[
        |F_m|=\tr(M+cc^\top)-\tr(M)=\tr(cc^\top)=2.
\]
Since all signless Laplacian eigenvalues of paths lie in \([0,4]\), the
interlacing intervals also lie in \([0,4]\).  Hence \(F_m\subset[0,4]\).

Let \(\alpha_j\) and \(\beta_j\) be the eigenvalues before and after the update, paired according to rank-one interlacing. Then $F_m=\bigcup_j[\alpha_j,\beta_j]$. Therefore, for every integer \(r\ge1\),
\[
\begin{aligned}
        \int_{F_m}y^r\,dy
        &=
        \sum_j\int_{\alpha_j}^{\beta_j}y^r\,dy 
        =
        \frac1{r+1}\sum_j(\beta_j^{r+1}-\alpha_j^{r+1})\\
        &=
        \frac{\tr Q(P_{m+1})^{r+1}-\tr Q(P_m)^{r+1}}{r+1}.
\end{aligned}
\]

We now bound this trace difference.  Let \(N_k\) be the unsigned incidence matrix of \(P_k\). Then $Q(P_k)=N_kN_k^\top$, and the subdivision graph \(S(P_k)\) is \(P_{2k-1}\), with biadjacency matrix \(N_k\). Hence the nonzero adjacency eigenvalues of \(P_{2k-1}\) are \(\pm\sqrt{q}\), where \(q\) runs over the nonzero eigenvalues of \(Q(P_k)\). Consequently
\[
        2\tr Q(P_k)^{r+1}
        =
        \tr A(P_{2k-1})^{2r+2}.
\]
Thus
\[
2\left(\tr Q(P_{m+1})^{r+1}-\tr Q(P_m)^{r+1}\right) =
        \tr A(P_{2m+1})^{2r+2}
        -
        \tr A(P_{2m-1})^{2r+2}.
\]
Viewing \(P_{2m-1}\) as the central subpath of \(P_{2m+1}\), this difference
counts the closed walks of length \(2r+2\) in \(P_{2m+1}\) which visit at least
one of the two new endpoints.

For each fixed new endpoint, the step-sequence map injects the closed walks
visiting that endpoint into the balanced \(\{\pm1\}\)-sequences of length
\(2r+2\), as in the proof of Lemma~\ref{lem:path-moments}. There are
\(\binom{2r+2}{r+1}\) such sequences. Taking a union bound over the two new endpoints gives
\[
        2\left(\tr Q(P_{m+1})^{r+1}-\tr Q(P_m)^{r+1}\right)
        \le
        2\binom{2r+2}{r+1},
\]
and the moment estimate \eqref{eq:lg-path-endpoint-moment} follows.

It remains to prove the two envelope bounds. We apply Lemma~\ref{lem:path-moment-envelope} to the interval set \(E=F_m\), using the moment constants supplied by \eqref{eq:lg-path-endpoint-moment}. On the intervals
\[
        [2,12/5],\qquad [12/5,14/5],\qquad [14/5,3],
\]
we use respectively the moments \(r=3,5,8\).  After clearing positive
denominators, the resulting inequalities are exactly the first three finite
inequalities verified in Appendix~\ref{app:linegraph-scalar-check}.  This
proves
\[
        J_{F_m}(t)\le L_{\rm av}(t)
        \qquad(2\le t\le3).
\]
On the intervals
\[
        [3,7/2],\qquad [7/2,58/15],\qquad [58/15,4],
\]
we use respectively the moments \(r=10,30,30\).  The corresponding exact
Bernstein checks in Appendix~\ref{app:linegraph-scalar-check} prove
\[
        J_{F_m}(t)\le L_{5,*}(t)
        \qquad(3\le t\le4).
\]
This completes the proof.
\end{proof}

The second local estimate compares a degree-\(3\)--degree-\(2\) non-bridge edge with the endpoint update for paths.

\begin{lemma}\label{lem:lg-support-edge}
Let \(G\) be a connected graph such that \(\Delta(G)\le3\) and no two
degree-\(3\) vertices are adjacent.  Let \(e=au\) be a non-bridge edge with
\(d_G(a)=3\) and \(d_G(u)=2\).  Put \(H=G-e\), and let \(m=|E(G)|\).  Then,
for every \(t\ge2\),
\[
        \Psi_t(G)-\Psi_t(H)
        \ge
        \Psi_t(P_{m+1})-\Psi_t(P_m).
\]
\end{lemma}

\begin{proof}
Let
\[
        b=e_a+e_u,
        \qquad
        M=Q(H),
        \qquad
        M_\theta=M+\theta bb^\top
        \quad(0\le\theta\le1).
\]
Since adding the edge \(au\) contributes
\((e_a+e_u)(e_a+e_u)^\top\) to the signless Laplacian, we have
\[
        Q(G)=M+bb^\top .
\]
By Lemma~\ref{lem:rank-one-trace-formula},
\begin{equation}\label{eq:lg-support-edge-trace}
        \frac12\bigl(\Psi_t(G)-\Psi_t(H)\bigr)
        =
        \int_0^1 b^\top(M_\theta-tI)_+b\,d\theta .
\end{equation}

Because \(H=G-e\), the vertices \(a\) and \(u\) are not adjacent in \(H\), and
\[
        d_H(a)=2,
        \qquad
        d_H(u)=1.
\]
Therefore
\[
        b^\top M b
        =
        b^\top Q(H)b
        =
        d_H(a)+d_H(u)
        =
        3,
\]
and hence
\[
        b^\top M_\theta b
        =
        b^\top M b+\theta\|b\|_2^4
        =
        3+4\theta.
\]
Applying Jensen's inequality to the convex function \(x\mapsto(x-t)_+\), we get
\[
        b^\top(M_\theta-tI)_+b
        \ge
        \|b\|_2^2
        \left(
        \frac{b^\top M_\theta b}{\|b\|_2^2}-t
        \right)_+  
        =
        2\left(\frac32+2\theta-t\right)_+ .
\]
Consequently, for \(2\le t\le3\),
\begin{equation}\label{eq:lg-support-edge-low-gain}
        \int_0^1 b^\top(M_\theta-tI)_+b\,d\theta
        \ge
        2\int_0^1\left(\frac32+2\theta-t\right)_+\,d\theta
        =
        L_{\rm av}(t).
\end{equation}

We next treat the range \(3\le t\le4\). Since \(0\le\theta\le1\), we have \(0\le M_\theta\le Q(G)\) entrywise. It is therefore enough to find a positive vector \(w\) such that \(Q(G)w\le5w\). Indeed, this gives \(M_\theta w\le5w\). If \(D=\operatorname{diag}(w)\), then \(D^{-1}M_\theta D\) is a non-negative matrix whose row sums are at most \(5\). Hence, by the standard row-sum bound for the spectral radius of a non-negative matrix~\cite[Theorem~8.1.22]{HornJohnson2012},
\[
        \rho(M_\theta)=\rho(D^{-1}M_\theta D)\le5.
\]

Now we define a positive vector \(w\) by
\[
        w_x=
        \begin{cases}
        1, & d_G(x)=3,\\
        2/3, & d_G(x)\le2.
        \end{cases}
\]
If \(d_G(x)=3\), then all neighbors of \(x\) have degree at most \(2\), by
the assumption that no two degree-\(3\) vertices are adjacent.  Hence
\[
        (Q(G)w)_x
        =
        d_G(x)w_x+\sum_{y\sim x}w_y
        =
        3\cdot1+3\cdot\frac23
        =
        5
        =
        5w_x.
\]
If \(d_G(x)\le2\), then \(w_x=2/3\), and every neighbor of \(x\) has weight at most \(1\).  Therefore
\[
        (Q(G)w)_x
        =
        d_G(x)w_x+\sum_{y\sim x}w_y
        \le
        d_G(x)\frac23+d_G(x)
        =
        \frac53d_G(x)
        \le
        \frac{10}{3}
        =
        5w_x.
\]
Thus \(Q(G)w\le5w\). Hence $\rho(M_\theta)\le5$. Since \(M_\theta\succeq0\), all eigenvalues of \(M_\theta\) are non-negative. Consequently,
\[
        \spec(M_\theta)\subset[0,5]
        \qquad(0\le\theta\le1).
\]

Let \(R\) be a shortest path from \(u\) to \(a\) in \(H\), and let \(v\) be the
predecessor of \(a\) on \(R\).  Since \(d_H(a)=2\), the vertex \(a\) has a
second neighbor in \(H\); denote it by \(a'\).  By the shortestness of \(R\),
we have \(a'\notin V(R)\), for otherwise the path from \(u\) to \(a'\) along
\(R\), followed by the edge \(a'a\), would give a shorter path from \(u\) to
\(a\).  Thus the subgraph \(S\) consisting of \(R\) together with the edge
\(aa'\) is a path with one additional leaf attached at \(a\), viewed relative
to the path \(R\).

We compare the fifth moment for this local subgraph with the fifth moment in
\(H\).  Embed \(Q(S)\) into the vertex space of \(H\) by adding zero rows and
columns outside \(S\).  Then
\[
        Q(S)+\theta bb^\top
        \le
        Q(H)+\theta bb^\top=M_\theta
\]
entrywise.  Since both matrices are entrywise non-negative, entrywise
domination is preserved under taking positive integer powers.  Therefore $\bigl(Q(S)+\theta bb^\top\bigr)^5
        \le
        M_\theta^5$ entrywise, and pairing with the non-negative vector \(b\) gives
\[
        b^\top M_\theta^5b
        \ge
        b^\top\bigl(Q(S)+\theta bb^\top\bigr)^5b .
\]

Let \(\ell\) be the length of \(R\), and let \(S_\ell\) denote the local path-plus-leaf graph consisting of the path
\[
        u=x_0-x_1-\cdots-x_\ell=a
\]
together with the extra leaf edge \(aa'\).  Put \(Q_\ell=Q(S_\ell)\) and \(P=bb^\top\). Then we know that
\[
        b^\top M_\theta^5b
        \ge
        b^\top(Q_\ell+\theta P)^5b .
\]
The latter quantity is a polynomial in \(\theta\).  If
\[
        s_j(\ell):=b^\top Q_\ell^j b,
        \qquad 0\le j\le5,
\]
then the word expansion of \((Q_\ell+\theta P)^5\), using \(P=bb^\top\), gives
\[
        b^\top(Q_\ell+\theta P)^5b
        =
        \sum_{m=0}^5\theta^m
        \sum_{\substack{r_0,\ldots,r_m\ge0\\
        r_0+\cdots+r_m=5-m}}
        \prod_{i=0}^m s_{r_i}(\ell).
\]
For the stable local model \(\ell=6\), a direct multiplication gives
\[
        (s_0(6),s_1(6),s_2(6),s_3(6),s_4(6),s_5(6))
        =
        (2,3,8,24,76,249).
\]
Substituting these values into the preceding formula gives
\[
        b^\top(Q_6+\theta P)^5b
        =
        64\theta^5+240\theta^4+472\theta^3
        +603\theta^2+512\theta+249.
\]
Only walks of length at most five contribute to this fifth moment, so the value is stable for all \(\ell\ge6\). For the remaining cases \(\ell=2,3,4,5\), the exact finite check in Appendix~\ref{app:linegraph-scalar-check} shows that $b^\top(Q_\ell+\theta P)^5b-q_*(\theta)$ has non-negative coefficients. Therefore, for every possible length of \(R\),
\[
        b^\top M_\theta^5b\ge q_*(\theta),
\]
where
\begin{equation}\label{eq:lg-support-edge-fifth-moment}
        q_*(\theta)
        =
        64\theta^5+240\theta^4+472\theta^3
        +603\theta^2+512\theta+249.
\end{equation}

For \(3\le t\le4\) and \(0\le y\le5\), we claim that
\[
        (y-t)_+
        \ge
        \frac{y^5-t^4y}{K_*(t)}.
\]
Indeed, the right-hand side is non-positive for \(y\le t\).  If \(t<y\le5\),
then
\[
        \frac{y^5-t^4y}{y-t}
        =
        y(y+t)(y^2+t^2),
\]
and the last expression is increasing in \(y\ge0\), hence is at most $5(5+t)(25+t^2)=K_*(t)$. Since \(\spec(M_\theta)\subset[0,5]\), functional calculus gives
\[
        (M_\theta-tI)_+
        \succeq
        \frac{M_\theta^5-t^4M_\theta}{K_*(t)}.
\]
Pairing with \(b\), using \(b^\top M_\theta b=3+4\theta\), and then using~\eqref{eq:lg-support-edge-fifth-moment}, we obtain
\[
        b^\top(M_\theta-tI)_+b
        \ge
        \frac{\bigl[q_*(\theta)-t^4(3+4\theta)\bigr]_+}{K_*(t)}.
\]
Here we also used the fact that the left-hand side is non-negative.

Therefore, by \eqref{eq:lg-support-edge-trace},
\[
        \frac12\bigl(\Psi_t(G)-\Psi_t(H)\bigr)
        \ge
        \frac1{K_*(t)}
        \int_0^1
        \bigl[q_*(\theta)-t^4(3+4\theta)\bigr]_+\,d\theta .
\]
Writing \(z=1-\theta\), the polynomial inside the positive part is
\[
        A_*(t)-B_*(t)z+4099z^2-2072z^3+560z^4-64z^5 .
\]
Let \(h(t):=A_*(t)/B_*(t)\).  On \(3\le t\le4\), we have
\(A_*(t)>0\), \(B_*(t)>0\), and
\[
h(t)\le h_{\max}:=\frac{1573}{4090}.
\]
The exact verification of this bound is recorded in
Appendix~\ref{app:linegraph-scalar-check}.  Since the function
\(z\mapsto 4099-2072z-64z^3\) is decreasing on \([0,\infty)\), the rational
inequality
\[
4099-2072h_{\max}-64h_{\max}^3>3000
\]
implies that, for \(0\le z\le h(t)\),
\[
4099z^2-2072z^3+560z^4-64z^5
\ge
3000z^2 .
\]

Therefore
\[
\begin{aligned}
        \int_0^1
        \bigl[q_*(\theta)-t^4(3+4\theta)\bigr]_+\,d\theta
        &\ge
        \int_0^{A_*/B_*}
        \bigl(A_*(t)-B_*(t)z+3000z^2\bigr)\,dz  \\
        &=
        \frac{A_*(t)^2}{2B_*(t)}
        +1000\frac{A_*(t)^3}{B_*(t)^3}.
\end{aligned}
\]
The exact algebra verifying \(h(t)\le1573/4090\) and the displayed rational lower bound is recorded in Appendix~\ref{app:linegraph-scalar-check}. Hence
\begin{equation}\label{eq:lg-support-edge-high-gain}
        \frac12\bigl(\Psi_t(G)-\Psi_t(H)\bigr)
        \ge
        L_{5,*}(t)
        \qquad(3\le t\le4).
\end{equation}

By Lemma~\ref{lem:lg-path-endpoint-envelope}, the path half-increment
\begin{equation}\label{eq:path-half-increment_v1}
        \frac12\bigl(\Psi_t(P_{m+1})-\Psi_t(P_m)\bigr)=J_{F_m}(t)
\end{equation}
is at most \(L_{\rm av}(t)\) for \(2\le t\le3\), and at most \(L_{5,*}(t)\) for \(3\le t\le4\). Combining these bounds with~\eqref{eq:lg-support-edge-trace}, \eqref{eq:lg-support-edge-low-gain}, and~\eqref{eq:lg-support-edge-high-gain}, we obtain
\[
        \Psi_t(G)-\Psi_t(H)
        \ge
        \Psi_t(P_{m+1})-\Psi_t(P_m)
\]
for \(2\le t\le4\).

Finally, if \(t\ge4\), then the signless Laplacian spectra of \(P_m\) and
\(P_{m+1}\) are contained in \([0,4]\), so
\[
        \Psi_t(P_{m+1})-\Psi_t(P_m)=0.
\]
The left-hand side is non-negative by the spectral-shift interval formula for
the positive rank-one update \(M\rightsquigarrow M+bb^\top\).  Hence the same
inequality holds for all \(t\ge4\), and the proof is complete.
\end{proof}

The terminal object in the deletion argument is a cycle. We record the comparison needed to exclude this terminal case.

\begin{lemma}\label{lem:lg-cycle-terminal}
For every \(m\ge3\) and every \(t\ge2\),
\[
        \Psi_t(C_m)\ge \Psi_t(P_{m+1}).
\]
\end{lemma}

\begin{proof}
Put \(x=t-2\ge0\). Since \(C_m\) is \(2\)-regular, $ Q(C_m)=2I+A(C_m)$. Therefore
\[
        \Psi_t(C_m)
        =
        \sum_i\bigl(\lambda_i(C_m)-x\bigr)_+^2,
\]
where \(\lambda_i(C_m)\) denotes the adjacency eigenvalues of \(C_m\).

We next identify the path side.  Let \(N\) be the unsigned incidence matrix of
\(P_{m+1}\).  Then
\[
        NN^\top=Q(P_{m+1}),
        \qquad
        N^\top N=2I+A(\mathcal L(P_{m+1}))=2I+A(P_m).
\]
The matrices \(NN^\top\) and \(N^\top N\) have the same nonzero eigenvalues. Thus the multiset
\[
\{\lambda_i(Q(P_{m+1}))-2\}
\] 
consists of the adjacency eigenvalues of \(P_m\), together with one additional value \(-2\). Since \(x\ge0\), this additional value contributes nothing to the positive part. Hence the desired inequality is equivalent to
\begin{equation}\label{eq:lg-cycle-terminal-adj-comparison}
        \sum_{\lambda_i(C_m)>0}\bigl(\lambda_i(C_m)-x\bigr)_+^2
        \ge
        \sum_{\lambda_i(P_m)>0}\bigl(\lambda_i(P_m)-x\bigr)_+^2 .
\end{equation}

Let \(c_1\ge c_2\ge\cdots\) and \(p_1\ge p_2\ge\cdots\) be the squared
positive adjacency eigenvalues of \(C_m\) and \(P_m\), respectively, padded by
zeroes.  We claim that $(c_i)\succ_w(p_i)$, that is,
\[
        \sum_{i=1}^k c_i\ge \sum_{i=1}^k p_i
        \qquad(k\ge1).
\]
Indeed, the adjacency eigenvalues of \(C_m\) are
\[
        2\cos\frac{2\pi j}{m},
        \qquad j=0,\ldots,m-1,
\]
and the adjacency eigenvalues of \(P_m\) are
\[
        2\cos\frac{i\pi}{m+1},
        \qquad i=1,\ldots,m.
\]
Thus, in the relevant positive range,
\[
        c_i=4\cos^2\frac{2\pi\lfloor i/2\rfloor}{m},
        \qquad
        p_i=4\cos^2\frac{i\pi}{m+1}.
\]

For the even partial sums, it is enough to compare the pairs
\((c_{2h-1},c_{2h})\) and \((p_{2h-1},p_{2h})\).  The required pair
comparison is
\[
\cos^2\frac{2(h-1)\pi}{m}
 +\cos^2\frac{2h\pi}{m}  \ge
 \cos^2\frac{(2h-1)\pi}{m+1}
 +\cos^2\frac{2h\pi}{m+1}.
\]
This is proved by the same angle argument as in Lemma~\ref{lem:cycle-pair-ineq}, with \(\alpha=2\pi/m\) and \(\beta=\pi/(m+1)\). Summing these pair inequalities gives all even partial-sum inequalities.

For the odd partial sums, after the preceding even partial sums it remains
only to compare the next term.  In the relevant range, $\frac{2h}{m}<\frac{2h+1}{m+1}$, and both angles lie in \([0,\pi/2]\).  Since \(\cos^2\) is decreasing on \([0,\pi/2]\), we have
\[
        4\cos^2\frac{2h\pi}{m}
        \ge
        4\cos^2\frac{(2h+1)\pi}{m+1}.
\]
This proves the odd partial-sum inequalities, and hence
\((c_i)\succ_w(p_i)\).

For fixed \(x\ge0\), the function $f_x(y):=(\sqrt y-x)_+^2$ is increasing and convex on \([0,\infty)\). Indeed, it vanishes on
\([0,x^2]\), while for \(y>x^2\),
\[
        f_x'(y)=1-\frac{x}{\sqrt y}\ge0,
        \qquad
        f_x''(y)=\frac{x}{2y^{3/2}}\ge0.
\]
By the standard characterization of weak submajorization by increasing convex functions \cite[Section~A.4]{MarshallOlkinArnold2011}, we have
\[
        \sum_i f_x(c_i)\ge \sum_i f_x(p_i).
\]
This is exactly \eqref{eq:lg-cycle-terminal-adj-comparison}, and the lemma follows.
\end{proof}

These local ingredients give the sharp edge-count signless Laplacian
stop-loss comparison.

\begin{theorem}\label{thm:edge-count-signless-stoploss}
Let \(G\) be a connected graph with \(m\ge1\) edges. Then, for every
\(t\ge2\),
\[
        \Psi_t(G)\ge \Psi_t(P_{m+1}).
\]
\end{theorem}

\begin{proof}
Assume that the assertion fails.  Choose a connected counterexample \(G\)
with the minimum possible number \(m\) of edges, and fix \(t\ge2\) such that
\[
        \Psi_t(G)<\Psi_t(P_{m+1}).
\]
The tree case is already contained in
Remark~\ref{rem:laplacian-stoploss-form}, because a tree with \(m\) edges has
\(m+1\) vertices.  Hence \(G\) is not a tree, and therefore \(G\) contains a
cycle.

We first show that every edge \(xy\in E(G)\) satisfies
\begin{equation}\label{eq:edge-count-degree-sum}
        d_G(x)+d_G(y)\le5 .
\end{equation}
Suppose, to the contrary, that \(e=xy\) satisfies \(d_G(x)+d_G(y)\ge6\), and put \(H=G-e\). Let \(b=e_x+e_y\), put \(M=Q(H)\), and set
\(M_\theta=M+\theta bb^\top\) for \(0\le\theta\le1\). Then
\(Q(G)=Q(H)+bb^\top\). Moreover, since \(x\) and \(y\) are not adjacent in \(H\),
\[
        b^\top Q(H)b
        =
        d_H(x)+d_H(y)
        =
        d_G(x)+d_G(y)-2
        \ge4.
\]
By Lemma~\ref{lem:rank-one-trace-formula},
\[
        \Psi_t(G)-\Psi_t(H)
        =
        2\int_0^1 b^\top(M_\theta-tI)_+b\,d\theta .
\]
Moreover,
\[
        b^\top M_\theta b
        =
        b^\top Q(H)b+\theta\|b\|_2^4
        \ge
        4+4\theta.
\]
Applying Jensen's inequality to the convex function \(s\mapsto(s-t)_+\), we get
\[
        b^\top(M_\theta-tI)_+b
        \ge
        \|b\|_2^2
        \left(
        \frac{b^\top M_\theta b}{\|b\|_2^2}-t
        \right)_+
        \ge
        2(2+2\theta-t)_+ .
\]
Therefore
\begin{equation}\label{eq:edge-count-large-degree-gain}
        \Psi_t(G)-\Psi_t(H)
        \ge
        4\int_0^1(2+2\theta-t)_+\,d\theta
        =
        2\mathcal I_{[2,4]}(t).
\end{equation}

If \(e\) is not a bridge, then \(H\) is connected with \(m-1\) edges.  By the
minimality of \(G\),
\[
        \Psi_t(H)\ge\Psi_t(P_m).
\]
The endpoint update \(P_m\rightsquigarrow P_{m+1}\) has spectral-shift interval set contained in \([0,4]\) and of total length \(2\). Hence, by~\eqref{eq:path-half-increment_v1}, Lemma~\ref{lem:spectral-shift-packing} and Lemma~\ref{lem:lg-path-endpoint-envelope},
\[
        \Psi_t(P_{m+1})-\Psi_t(P_m)
        \le
        2\mathcal I_{[2,4]}(t).
\]
Combining this with
\eqref{eq:edge-count-large-degree-gain}, we obtain
\[
        \Psi_t(G)
        \ge
        \Psi_t(P_{m+1}),
\]
contradicting the choice of \(G\).

Now suppose that \(e\) is a bridge.  Let \(H_1\) and \(H_2\) be the two
components of \(H\), and let \(m_i=|E(H_i)|\).  By minimality, together with
the trivial isolated-vertex case, we have
\[
        \Psi_t(H_i)\ge\Psi_t(P_{m_i+1})
        \qquad(i=1,2).
\]
Since \(\Psi_t\) is additive over disjoint unions,
\[
        \Psi_t(H)
        \ge
        \Psi_t(P_{m_1+1})+\Psi_t(P_{m_2+1}).
\]
The graph \(P_{m+1}\) is obtained from
\(P_{m_1+1}\sqcup P_{m_2+1}\) by joining two endpoints, since
\(m=m_1+m_2+1\).  This signless Laplacian update is again a rank-one update
of trace \(2\), and both the initial and final path spectra are contained in
\([0,4]\).  Thus its cost is at most
\(2\mathcal I_{[2,4]}(t)\). Together with
\eqref{eq:edge-count-large-degree-gain}, this again gives
\[
        \Psi_t(G)\ge\Psi_t(P_{m+1}),
\]
a contradiction. Therefore~\eqref{eq:edge-count-degree-sum} holds.

We next record two structural consequences.  First, \(\Delta(G)\le3\).  Indeed,
if some vertex \(v\) had degree at least \(4\), then, since \(G\) is connected
and contains a cycle, \(v\) would have a neighbor \(w\) with \(d_G(w)\ge2\)
along a path from \(v\) to a cycle.  The edge \(vw\) would then satisfy
\[
        d_G(v)+d_G(w)\ge4+2=6,
\]
contradicting~\eqref{eq:edge-count-degree-sum}. Second, no two degree-\(3\) vertices are adjacent, because such an edge would have degree-sum \(6\).

If \(G\) is not a cycle, then some cycle vertex \(a\) has degree \(3\).  Its
two cycle neighbors have degree \(2\), since no two degree-\(3\) vertices are
adjacent.  Choose one of them, say \(u\), and let \(e=au\).  This edge lies on
a cycle, and hence is not a bridge.  Lemma~\ref{lem:lg-support-edge} gives
\[
        \Psi_t(G)-\Psi_t(G-e)
        \ge
        \Psi_t(P_{m+1})-\Psi_t(P_m).
\]
Since \(G-e\) is connected with \(m-1\) edges, minimality gives $\Psi_t(G-e)\ge\Psi_t(P_m)$. Adding the last two inequalities yields
\[
        \Psi_t(G)\ge\Psi_t(P_{m+1}),
\]
again contradicting the choice of \(G\).

Therefore the only remaining minimal counterexample would be a cycle. But cycles are excluded by Lemma~\ref{lem:lg-cycle-terminal}.  This final contradiction proves the theorem.
\end{proof}

The edge-count comparison implies the sharp positive \(p\)-energy bound for line graphs.

\begin{theorem}\label{thm:line-graph-positive-energy}
Let \(G\) be a connected graph with \(m\ge1\) edges, and let
\(\mathcal L(G)\) be its line graph.  Then, for every real number \(p\ge2\),
\[
        \mathcal E_p^+\bigl(\mathcal L(G)\bigr)
        \ge
        \mathcal E_p^+(P_m).
\]
\end{theorem}

\begin{proof}
Let \(N_G\) be the unsigned vertex-edge incidence matrix of \(G\).  Then
\[
        N_GN_G^\top=Q(G),
        \qquad
        N_G^\top N_G=2I+A\bigl(\mathcal L(G)\bigr).
\]
The matrices \(N_GN_G^\top\) and \(N_G^\top N_G\) have the same nonzero
eigenvalues.  Therefore the positive adjacency eigenvalues of
\(\mathcal L(G)\) are exactly the positive numbers among
\(\lambda_i(Q(G))-2\).  Hence, for every \(p>0\),
\begin{equation}\label{eq:line-graph-positive-energy-q-shift}
        \mathcal E_p^+\bigl(\mathcal L(G)\bigr)
        =
        \sum_i\bigl(\lambda_i(Q(G))-2\bigr)_+^p .
\end{equation}

For \(p=2\), Theorem~\ref{thm:edge-count-signless-stoploss}, applied with
threshold \(2\), gives
\[
        \sum_i\bigl(\lambda_i(Q(G))-2\bigr)_+^2
        \ge
        \sum_i\bigl(\lambda_i(Q(P_{m+1}))-2\bigr)_+^2.
\]

Now assume \(p>2\). We shall use the elementary representation
\[
        y_+^p
        =
        \frac{p(p-1)(p-2)}2 \int_0^\infty (y-s)_+^2s^{p-3}\,ds
        \qquad(y\in\mathbb R).
\]
For every \(s\ge0\), the threshold \(2+s\) is at least \(2\), and hence
Theorem~\ref{thm:edge-count-signless-stoploss} gives
\[
        \sum_i\bigl(\lambda_i(Q(G))-2-s\bigr)_+^2
        \ge
        \sum_i\bigl(\lambda_i(Q(P_{m+1}))-2-s\bigr)_+^2 .
\]
Multiplying this inequality by the positive weight \(\frac{p(p-1)(p-2)}2 s^{p-3}\) and integrating over \(s\in[0,\infty)\), we obtain
\[
        \sum_i\bigl(\lambda_i(Q(G))-2\bigr)_+^p
        \ge
        \sum_i\bigl(\lambda_i(Q(P_{m+1}))-2\bigr)_+^p.
\]
Together with the \(p=2\) case, this inequality holds for every \(p\ge2\).

Finally, \(\mathcal L(P_{m+1})=P_m\).  Applying~\eqref{eq:line-graph-positive-energy-q-shift} to the path \(P_{m+1}\), we get
\[
        \sum_i\bigl(\lambda_i(Q(P_{m+1}))-2\bigr)_+^p
        =
        \mathcal E_p^+(P_m).
\]
Combining this equality with~\eqref{eq:line-graph-positive-energy-q-shift} proves
\[
        \mathcal E_p^+\bigl(\mathcal L(G)\bigr)
        \ge
        \mathcal E_p^+(P_m).
\qedhere\]
\end{proof}

This theorem has the following square-energy interpretation.

\begin{remark}\label{rem:line-graph-claw-free-square-energy}
Recall that \(s^+(H)=\mathcal E_2^+(H)\) denotes the positive square energy of a graph \(H\). Taking $p=2$ in Theorem~\ref{thm:line-graph-positive-energy} gives
\[
        s^+\bigl(\mathcal L(G)\bigr)
        \ge
        s^+(P_m)=m-1=|V(\mathcal L(G))|-1.
\]
Thus Theorem~\ref{thm:line-graph-positive-energy} gives the line-graph case of the positive square-energy lower bound. This should be compared with the recent theorem of Akbari, Kumar, Mohar, Pragada and Zhang, who prove a stronger \(p=2\) refinement for connected claw-free graphs with maximum degree at least \(3\)~\cite[Theorem~1.1]{AkbariKumarMoharPragadaZhang2025Refinement}. Since line graphs are claw-free, their result applies to line graphs satisfying this maximum-degree hypothesis in a broader graph class. The point of Theorem~\ref{thm:line-graph-positive-energy} is different: within the class of line graphs, it gives the same sharp path lower bound for every real exponent \(p\ge2\).
\end{remark}

\appendix

\section{Exact line-graph scalar and local-moment checks}
\label{app:linegraph-scalar-check}

This appendix records the finite checks used in
Lemmas~\ref{lem:lg-path-endpoint-envelope} and~\ref{lem:lg-support-edge}.  The
checks are exact rational computations.  The auxiliary SageMath checks were
carried out in SageMath~\cite{SageMath10.8}.

We use the notation
\[
        A_*(t)=2140-7t^4,
        \qquad
        B_*(t)=4414-4t^4,
        \qquad
        K_*(t)=5(t^3+5t^2+25t+125),
\]
\[
        L_{\rm av}(t)=\frac12\left(\frac72-t\right)^2,
        \qquad
        \widehat L_*(t)=\frac12A_*(t)^2B_*(t)^2+1000A_*(t)^3,
        \qquad
        M_r:=\frac1{r+1}\binom{2r+2}{r+1}.
\]
The scalar inequalities reduce to the positivity of the following six
polynomials on the stated intervals:
\[
\begin{array}{c|c}
\text{polynomial} & \text{interval} \\ \hline
 t^2L_{\rm av}(t)-M_3\dfrac{2^2}{3^3} & [2,12/5] \\[2mm]
 t^4L_{\rm av}(t)-M_5\dfrac{4^4}{5^5} & [12/5,14/5] \\[2mm]
 t^7L_{\rm av}(t)-M_8\dfrac{7^7}{8^8} & [14/5,3] \\[2mm]
 t^9\widehat L_*(t)-M_{10}\dfrac{9^9}{10^{10}}B_*(t)^3K_*(t) & [3,7/2] \\[2mm]
 t^{29}\widehat L_*(t)-M_{30}\dfrac{29^{29}}{30^{30}}B_*(t)^3K_*(t) & [7/2,58/15] \\[2mm]
 \widehat L_*(t)-M_{30}\dfrac{4-t}{4^{30}}B_*(t)^3K_*(t) & [58/15,4].
\end{array}
\]

We certify these inequalities by Bernstein coefficients.  More precisely, if
\(P\in\QQ[t]\), \(a<b\) are rational, and \(m\ge\deg P\), the degree-\(m\)
Bernstein coefficients \(b_0,\ldots,b_m\) of \(P\) on \([a,b]\) are defined by
\[
        P((1-x)a+xb)
        =
        \sum_{k=0}^m b_k\binom{m}{k}x^k(1-x)^{m-k}.
\]
If \(b_k\ge0\) for all \(k\), then \(P(t)\ge0\) on \([a,b]\).  In the present
case, the SageMath verification script computes all Bernstein coefficients
over \(\QQ\) and checks that they are strictly positive for all six rows of the
table.

We next record the finite local fifth-moment checks used in
Lemma~\ref{lem:lg-support-edge}.  For the stable local model one has
\[
        q_*(\theta)
        =
        64\theta^5+240\theta^4+472\theta^3
        +603\theta^2+512\theta+249.
\]
For the exceptional lengths \(\ell=2,3,4,5\), the exact differences are
\[
\begin{aligned}
b^\top(Q_2+\theta P)^5b-q_*(\theta)
    &=64\theta^3+192\theta^2+256\theta+147,\\
b^\top(Q_3+\theta P)^5b-q_*(\theta)
    &=24\theta^2+68\theta+70,\\
b^\top(Q_4+\theta P)^5b-q_*(\theta)
    &=8\theta+18,\\
b^\top(Q_5+\theta P)^5b-q_*(\theta)
    &=2.
\end{aligned}
\]
All coefficients are non-negative.  For \(\ell=6\), one has equality,
\[
        b^\top(Q_6+\theta P)^5b=q_*(\theta),
\]
and this value is stable for all \(\ell\ge6\), since walks of length at most
five cannot see farther along the path.

Finally, we record the elementary rational identities used in the
high-threshold part of Lemma~\ref{lem:lg-support-edge}.  First,
\[
\begin{aligned}
&q_*(1-z)-t^4(3+4(1-z)) \\
&\qquad =A_*(t)-B_*(t)z+4099z^2-2072z^3+560z^4-64z^5.
\end{aligned}
\]
Let \(h(t)=A_*(t)/B_*(t)\).  On \(3\le t\le4\), both \(A_*(t)\) and \(B_*(t)\)
are positive, since
\[
        A_*(4)=348>0,
        \qquad
        B_*(4)=3390>0.
\]
Moreover,
\[
        \frac{1573}{4090}B_*(t)-A_*(t)
        =
        \frac{11169}{2045}(t-3)(t+3)(t^2+9)\ge0
        \qquad(3\le t\le4).
\]
Therefore
\[
        0<h(t)\le \frac{1573}{4090}
        \qquad(3\le t\le4).
\]
The remaining rational margin is
\[
4099-2072\frac{1573}{4090}
-64\left(\frac{1573}{4090}\right)^3-3000
=
\frac{2552631003539}{8552241125}>0.
\]
Consequently,
\[
        4099-2072z-64z^3>3000
        \qquad
        \left(0\le z\le\frac{1573}{4090}\right),
\]
which is the rational lower bound used in the proof.

The full SageMath script used to check the Bernstein certificates and the
finite local fifth-moment coefficient inequalities is
\[
        \texttt{applications/verify\_line\_graph.sage}
\]
in the public GitHub repository~\cite{Git}.  The corresponding output log is
\[
        \texttt{applications/verify\_line\_graph\_output.txt}.
\]
The script was run with SageMath~10.8.  All computations are performed over
\(\QQ\), with no floating-point arithmetic.

\section*{Acknowledgments and AI disclosure}

The authors thank Clive Elphick for helpful comments and suggestions. Q.~Tang was partially supported by the National Key Research and Development Program of China under grant 2023YFA1010203. Y.~Liu  was partially supported by Beijing Natural Science Foundation under grant QY26408.

The authors used GPT-5.5 Pro for limited
assistance with drafting, code generation, and presentation.  In particular,
AI assistance partly drafted the code included in the appendix, and the authors
then checked, edited, and verified it.

The mathematical ideas, proof strategy, and final arguments are due to the
authors.  The authors used AI tools only as aids for drafting, coding, and
presentation.  The authors are fully responsible for all statements, proofs,
computations, and figures in the paper.

\end{document}